\documentclass[10pt]{article}

\usepackage{mathptmx}
\usepackage{eucal}
\usepackage{amsmath}
\usepackage{amscd}
\usepackage{amssymb}
\usepackage{amsthm}
\usepackage{xspace}
\usepackage[all,tips]{xy}
\usepackage[dvips]{graphicx}
\usepackage{graphics}
\usepackage{epsfig}
\usepackage{verbatim}
\usepackage{syntonly}
\usepackage{hyperref}
\usepackage{amssymb, color, url}

\newtheorem{theorem}{Theorem}
\newtheorem{proposition}{Proposition}
\newtheorem{claim}{Claim}

\newtheorem{corollary}[theorem]{Corollary}
\newtheorem{lemma}[theorem]{Lemma}
\newtheorem{question}{Question}

\newtheorem*{theoremNL}{Theorem}

\input xy
\xyoption{all}

\DeclareMathOperator{\rk}{rk}

\DeclareMathOperator{\Z}{\mathbb{Z}}
\newcommand{\ffp}{C_{p}*C_{p}}
\newcommand{\Ab}{\mathrm{Ab}}

\newcommand{\G}{\mathcal{G}}

\usepackage{fancyhdr}

\fancypagestyle{plain}

\begin{document}
\bibliographystyle{plain}


\title{\textbf{Generalizing Magnus' characterization \\of free groups to some free products}}
\author{Khalid Bou-Rabee\thanks{University of Michigan, Email: \tt{khalidb@umich.edu}}~ and Brandon Seward\thanks{University of Michigan, Email: \tt{b.m.seward@gmail.com}}}

\maketitle


\begin{abstract}
A residually nilpotent group is \emph{$k$-parafree} if all of its lower central series quotients match those of a free group of rank $k$.
Magnus proved that $k$-parafree groups of rank $k$ are themselves free.
In this note we mimic this theory with finite extensions of free groups, with an emphasis on free products of the cyclic group $C_p$, for $p$ an odd prime.
We show that for $n \leq p$ Magnus' characterization holds for the $n$-fold free product $C_p^{*n}$ within the class of finite-extensions of free groups. Specifically, if $n \leq p$ and $G$ is a finitely generated, virtually free, residually nilpotent group having the same lower central series quotients as $C_p^{*n}$, then $G \cong C_p^{*n}$.
We also show that such a characterization does not hold in the class of finitely generated groups.
That is, we construct a rank $2$ residually nilpotent group $G$ that shares all its lower central series quotients with $\ffp$, but is not $\ffp$.
\end{abstract}
\vskip.1in
{\small{\bf keywords:}
\emph{parafree, lower central series, free products, residually nilpotent.}}

\section*{Introduction}

Let $C_m$ be the cyclic group of order $m$.
This note addresses whether free products of the form $C_p*C_p* \cdots *C_p$, denoted $C_p^{*n}$ for short, can be characterized by their lower central series quotients.
Recall that the \emph{lower central series} of a group $G$ is defined to be
\[
 \gamma_1(G) := G \text{ and }\gamma_k(G) := [G, \gamma_{k-1}(G)] \text{ for $k \geq 2$},
 \]
 where $[A,B]$ denotes the group generated by commutators of elements of $A$ with elements of $B$. The \emph{rank} of $G$, denoted $\rk(G)$, is the minimum size of a generating set of $G$.
In 1939, Magnus gave a beautiful characterization of free groups in terms of their lower central series quotients \cite{M1}.

\begin{theoremNL}[Magnus' characterization of free groups] \label{magnusTheorem}
Let $F_k$ be a nonabelian free group of rank $k$ and $G$ a group of rank $k$.
If $G/\gamma_i(G) \cong F_k/\gamma_i(F_k)$ for all $i$, then $G \cong F_k$.
\end{theoremNL}

Following this result, Hanna Neumann inquired whether it was possible for two residually nilpotent groups $G$ and $G'$ to have $G/\gamma_i(G) \cong G'/\gamma_i(G')$ for all $i$ without having $G \cong G'$ (see \cite{L1}).
Recall that a group $G$ is \emph{residually nilpotent} if $\cap_{k=1}^\infty \gamma_k(G) = \{1\}$.
Gilbert Baumslag \cite{baumslag-1967} gave a positive answer to this question (c.f. \cite{baumslag-1968}, \cite{bridson-grunewald}), thereby beginning the subject of parafree groups. A group $G$ is \emph{parafree} if:

\begin{enumerate}
\item $G$ is residually nilpotent, and
\item there exists a finitely generated free group $F$ with the property that $G/\gamma_i(G) \cong F/\gamma_i(F)$ for all $i$.
\end{enumerate}

By Magnus' Theorem, Baumslag's examples necessarily have rank different from the corresponding free group.
In \cite{bou}, the first author explored what happens when the role of a free group is played by the fundamental group of a closed surface of a given genus.
Our first main result is a Magnus Theorem for $n$-fold free products $C_p^{*n}$ within the class of finite-extensions of free groups.
The proof in Section \ref{ProofSection1} relies on a theorem of A.~Karrass, A.~Pietrowski, and D.~Solitar \cite{KPS73} (see also G.~P.~Scott \cite{S74}).

\begin{theorem} \label{MainTheorem3}
Let $p$ be an odd prime and let $n \leq p$.
Let $G$ be a residually nilpotent and finitely generated group such that
  $$G/\gamma_i(G) \cong C_p^{*n} / \gamma_i(C_p^{*n})$$
for every $i$. If $G$ is virtually free, then $G \cong C_p^{*n}$.
\end{theorem}

Our second main result, which we prove in Section \ref{ProofSection2}, shows that it is impossible to drop ``virtually free'' from Theorem \ref{MainTheorem3}. In fact, we construct rank two residually nilpotent groups that  have all the same lower central series quotients as $\ffp$ but are not $\ffp$.
These examples are analogous to Baumslag's parafree groups, where the role of free groups is replaced by $\ffp$.
Consequently, such groups are called \emph{para-$(\ffp)$ groups}.

\begin{theorem} \label{MainTheorem1}
Let $p$ be an odd prime.
There exist rank two groups $G_1$ and $G_2$, both not isomorphic to $\ffp$, such that
 $$G_1/\gamma_i(G_1) \cong G_2/\gamma_i(G_2) \cong \ffp / \gamma_i(\ffp)$$
 for all $i$. Further, $G_1$ is residually nilpotent and $G_2$ is not.
\end{theorem}

\noindent

\noindent
{\bf Remarks:}
First, the examples found in this note were discovered with the use of GAP \cite{GAP}. Second, it is natural to ask why we ignore the general case $C_m*C_n$ for arbitrary natural numbers $m,n$. We don't consider these groups because $C_m*C_n$ is residually nilpotent if and only if $m$ and $n$ are powers of a fixed prime $p$. 
The generalization of Theorem \ref{MainTheorem1} to the case $C_{p^l} * C_{p^k}$, for natural numbers $l,k$ and prime number $p$, is not hard, see Section \ref{generalizationSection}.
Finally, the paper is partially motivated by possible applications to three-manifold theory. Please see Section \ref{finalremarks} for remarks on this connection. 

\subsubsection*{Acknowledgements}

The first author would like to thank Tim Cochran for initiating this topic. The first author is also grateful to Benson Farb, Ben McReynolds, and Richard Canary for valuable discussions and support. The second author's research was supported by a National Science Foundation Graduate Research Fellowship. Finally, we are thankful to Tom Church for comments on a previous draft of this paper.

\section{Preliminaries and notation} \label{PrelimSection}

\paragraph{Known group theory results.}
We first list a couple of known results needed in the proofs of our main theorems.

\begin{lemma}[Lemma 5.9, page 350 in Magnus, Korass and Solitar \cite{MKS}]  \label{generatorlemma}
Let $G$ be a free nilpotent group of class $c$ and let $g_1, g_2,  \ldots \in G$ be elements whose projections to $G/[G, G]$ generate $G/[G,G]$.
Then $g_1, g_2,  \ldots$ generate $G$.
\end{lemma}

\begin{lemma}[Theorem 4.1 in \cite{MKS}]  \label{NormalFfpForm}
If $\gamma \neq 1$ is in $\ffp = \langle a, b : a^p, b^p \rangle$, then there exists a unique reduced sequence $g_1, \ldots, g_k$ such that
\begin{equation} \label{normalform}
\gamma = g_1 \cdots g_k,
\end{equation}
where $g_i$ are elements from $\{ a, a^{2}, \ldots, a^{p-1}, b, b^{2}, \ldots, b^{p-1} \}$.
In particular, if $\gamma^p = 1$, then $\gamma$ is conjugate to either $a$ or $b$.
\end{lemma}

\paragraph{Graph of groups.} We state here our notation for graphs of groups. A \emph{graph of groups} is a graph $\G$ where the vertices and edges of $\G$ are groups and for each edge $E$ of $\G$ there are monomorphisms $\phi_{E,0}$ and $\phi_{E,1}$ mapping $E$ into the vertex groups joined by $E$.

Let $T$ be a spanning tree for $\G$ and define the \emph{fundamental group} of $\G$, denoted $\pi_1(\G)$, to be the group generated by the vertex groups together with elements $\{Y_E : E \text{ an edge of } \G\}$ subject to the following conditions:
\begin{itemize}
\item $Y_{\overline{E}} = Y_E^{-1}$ if $\overline{E}$ is the edge $E$ with the reverse orientation;
\item $Y_E \phi_{E,0} (g) Y_E^{-1} = \phi_{E,1}(g)$ for every edge $E$ and every $g \in E$;
\item $Y_E = 1$ if $E$ is an edge of $T$.
\end{itemize}
This definition is independent of the choice of spanning tree.

\paragraph{Para-$\Gamma$ groups}
In the remainder of this section, we generalize Baumslag's definition of parafree groups to arbitrary groups. 
Let $\Gamma$ be a finitely generated residually nilpotent group.
A group $G$ is a \emph{weakly para-$\Gamma$ group} if $G/\gamma_k(G) \cong \Gamma / \gamma_k(\Gamma)$ for all $k \geq 1$.
If $G$ is weakly para-$\Gamma$ and residually nilpotent, we say that $G$ is a \emph{para-$\Gamma$ group}. Let $G = F/N$ be a weakly para-$\Gamma$ group where $F$ is a free group of rank $\rk(\Gamma)$.
Let $\Gamma = F/K$.
Then we have the following trichotomy for such groups $G$:
\begin{itemize}
\item[] Type I. There exists an isomorphism $\phi: F \to F$ such that $\phi(N) \geq K$.
\item[] Type II. There exists an isomorphism $\phi: F \to F$ such that $\phi(N) \lneq  K$.
\item[] Type III. $G$ is not of Type I or II.
\end{itemize}

We note that it is not clear, a priori, that there does not exist groups that are both Type I and Type II. This fact is a consequence of the next theorem, which is a slightly more general version of a theorem appearing in \cite{bou}.

\begin{proposition} \label{typeTheorem}
Groups of Type I must be $\Gamma$.
Further, groups of Type II are never para-$\Gamma$ groups.
\end{proposition}

\begin{proof}
We first show that groups of Type I must be isomorphic to $\Gamma$.
For the sake of a contradiction, suppose that $G$ is of Type I and is not isomorphic to $\Gamma$.
Let $F$ and $K$ be as in the definition of Type I groups.
By assumption, there exists an isomorphism $\phi: F \to F$ such that $\phi(N) \geq K$.
The isomorphism $\phi^{-1}$ induces a homomorphism $\rho_i : \Gamma/\gamma_i(\Gamma) \to G/ \gamma_i(G)$ which is surjective for all $i$.
As finitely generated nilpotent groups are Hopfian (see Section III.A.19 in \cite{harpe-2000}, for instance), the maps $\rho_i$ must be isomorphisms for all $i$.
On the other hand, since $G$ is not isomorphic to $\Gamma$, we must have some $\gamma \in \phi(N) - K$.
Further, $F/K= \Gamma$ is residually nilpotent, so there exists some $i$ such that $\gamma \neq 1$ in $\Gamma/ \gamma_i(\Gamma)$.
Since $\gamma \in \ker \rho_i$, we have a contradiction.

We now show that groups of Type II are never residually nilpotent.
For the sake of a contradiction, suppose that $G$ is a residually nilpotent group of Type II.
Let $F$ and $K$ be as in the definition of Type II groups.
By assumption, the map $\phi: F \to F$, induces a map $\psi: G \to \Gamma$ that is onto with non-trivial kernel.
Let $g \in \ker \psi$.
Since $G$ is residually nilpotent, there exists $i$ such that $g \notin \gamma_i(G)$.
Hence, the induced map $\rho_i: G/\gamma_i(G) \to \Gamma/\gamma_i(\Gamma)$ is onto but not bijective, which is impossible as finitely generated nilpotent groups are Hopfian.
\end{proof}

\section{Characterizing $C_p^{*n}$ by its lower central series quotients} \label{ProofSection1}

In this section we prove that if $G$ is a para-$C_p^{*n}$ group and $n \leq p$, then $G$ is isomorphic to $C_p^{*n}$ provided it is finitely generated and virtually free. The following lemma is a key ingredient to the proof.

\begin{lemma} \label{MainLemma1}
Let $p$ be an odd prime and let $G$ be a group that is para-$C_p^{*n}$. If $G$ is virtually free and finitely generated, then $[G,G]$ is a free group. Moreover, $[G, G]$ and $[C_p^{*n}, C_p^{*n}]$ have the same rank and same index in $G$ and $C_p^{*n}$, respectively.
\end{lemma}

\begin{proof}
Let $k \geq 2$ be a natural number.
The group $\Ab(\gamma_k(G))$ is a finitely generated
abelian group, since $\gamma_k(G)$ is finite index in $G$. So
$$
\Ab(\gamma_k(G)) = \prod_{q} Q_q \times \prod_{i=1}^N \Z,
$$
where $N \in \mathbb{N}$ and each $Q_q$ is a finite $q$-group where the product is taken over at most finitely many primes $q$.
We aim to show that $N = \rk(\Ab(\gamma_k(C_p^{*n})))$ and $Q_p = 1$.
Let $V$ be the verbal subgroup of $\Ab(\gamma_k(G))$ given by
$$
V = \left< a : \exists b, b^p = a\right>.
$$
This is a characteristic subgroup of $\Ab(\gamma_k(G))$ and $\Ab(\gamma_k(G))/V = \prod_{i=1}^{N+ \rk(Q_p)}C_p$.
Set $K$ to be the pullback of the subgroup $V$ to $G$. The subgroup $K$ of $G$ is normal and $G/K$ is a $p$-group. Further, the map $\phi: G \to G/K$ satisfies
$$
\gamma_k(G/K) = \phi(\gamma_k(G)) = \Ab(\gamma_k(G))/V.
$$
Since $G$ is para-$C_p^{*n}$, there exists a map
$$
\psi : C_p^{*n} \to G/K,
$$
and $\psi(\gamma_k(C_p^{*n})) = \gamma_k(G/K) = \Ab(\gamma_k(G)) / V$.
Hence
 $$N + \rk(Q_p) \leq \rk(\Ab(\gamma_k(C_p^{*n}))) = R,$$
where $R$ is defined to be the rank of $\Ab(\gamma_k(C_p^{*n}))$. To get a reverse inequality, let $W$ be the characteristic subgroup of $\Ab(\gamma_k(C_p^{*n}))$ generated by all $p^r$ powers of elements.
That is
$$
W = \left< a : \exists b, b^{p^r} = a \right>.
$$
Then $\Ab(\gamma_k(C_p^{*n})) / W = \prod_{i=1}^{R} C_{p^r}$, since we know that $\gamma_k(C_p^{*n})$ is a free group.
Let $W'$ be the pullback of $W$ to $C_p^{*n}$. The group $W'$ is normal in $C_p^{*n}$ and $C_p^{*n}/W'$ is a $p$-group.
Since $G$ is para-$C_p^{*n}$ there must exist a surjective map $\phi : G \to \ffp/W'$.
Hence,
$$
\phi(\gamma_k(G)) = \gamma_k(C_p^{*n} / W') = \Ab(\gamma_k(C_p^{*n}))/W =  \prod_{i=1}^{R} C_{p^r}.
$$
Picking $r$ large enough gives the inequality $N \geq R$, so $N = R$ and $\rk(Q_p) = 0$.

We conclude that the torsion elements in $\Ab(\gamma_k(G))$ do not have $p$ power order.
We claim that this implies that $[G,G]$ must be torsion-free.
Suppose not, let $\gamma \in [G,G]$ be a nontrivial element with finite-order $n$.
Suppose that $p | n$, then $\gamma' = \gamma^{n/p}$ has order a power of $p$.
Then let $r$ be the first natural number such that $\gamma' \in \gamma_r(G) - \gamma_{r+1}(G)$.
Then $\gamma' \notin [\gamma_r(G), \gamma_r(G)]$, so $\gamma'$ survives in $\Ab(\gamma_r(G))$. It follows that $\gamma$ cannot have order that is divisible by $p$.
But $G$ is residually $p$-finite, being para-$C_p^{*n}$, hence $\gamma$'s power must be a power of $p$, a contradiction. So $[G, G]$ is torsion free. Also $[G, G]$ and $[C_p^{*n}, C_p^{*n}]$ have the same index in $G$ and $C_p^{*n}$, respectively, since $\Ab(G) = \Ab(C_p^{*n})$. Finally, by Stallings \cite{S68} $[G, G]$ and $[C_p^{*n}, C_p^{*n}]$ are free groups as they are torsion free and virtually free. These two groups have the same rank since the above argument shows that
$$\rk([G, G]) = \rk(\Ab([G, G])) = \rk(\Ab([C_p^{*n}, C_p^{*n}])) = \rk([C_p^{*n}, C_p^{*n}]).$$
\end{proof}

We will soon show that every finitely generated, virtually free, para-$C_p^{*n}$ group $G$ is the fundamental group of a finite graph of groups where the underlying graph is a tree. The following two lemmas will allow us to describe this tree more explicitly. In our arguments we implicitly use the well known fact that every tree with $k$ vertices has $k-1$ edges.

\begin{lemma} \label{LEM INDX}
Let $\G$ be a tree, and let $u$ be a vertex in $\G$. Then there are enumerations $v_1, v_2, \ldots, v_k$ and $e_1, e_2, \ldots, e_{k-1}$ of the vertices and edges of $\G$, respectively, such that $v_k = u$ and $v_i$ is an endpoint of $e_i$ for all $1 \leq i < k$. Furthermore, for every $1 \leq i < k$ we have that $v_i$ is the endpoint of $e_i$ which is furthest from $v_k = u$.
\end{lemma}

\begin{proof}
Let $e_1, e_2, \ldots, e_{k-1}$ be any enumeration of the edges of $\G$. For each $1 \leq i < k$, let $v_i$ be the endpoint vertex of $e_i$ which is further from $u$. Since $\G$ is a tree, we have $v_i \neq v_j$ for $i \neq j$. Now set $v_k = u$. Since $\G$ has $k$ vertices and $v_1, v_2, \ldots, v_k$ are distinct, this must be an enumeration of the vertices of $\G$. Clearly the enumerations of the vertices and the edges have the desired properties.
\end{proof}

\begin{lemma} \label{LEM ABTREE}
If $\G$ is a finite tree of finite abelian groups then the order of $\Ab(\pi_1(\G))$ is the product of the orders of the vertex groups divided by the product of the orders of the edge groups.
\end{lemma}

\begin{proof}
The proof is by induction on $k$, the number of vertices of $\G$. If $k = 1$ then the claim is clearly true. Now suppose the claim is true for all $m < k$. Suppose $\G$ has $k$ vertices. Pick a vertex $V$ of $\G$ having degree one, and let $E$ be the unique edge incident to $V$. Let $\G'$ be the tree obtained by removing $V$ and $E$ from $\G$. By the inductive hypothesis, we have that the order of $\Ab(\pi_1(\G'))$ is the product of the orders of its vertices divided by the product of the orders of its edges. So it will suffice to show that:
$$|\Ab(\pi_1(\G))| = |\Ab(\pi_1(\G'))| \cdot \frac{|V|}{|E|}$$
Set $G' = \pi_1(\G')$. Notice that $E$ is a subgroup of one of the vertex groups of $\G'$ and therefore $E$ is a subgroup of $G'$. Since $[G', G']$ is a free group and $E$ is finite, we have $E \cap [G', G'] = \{1\}$. Therefore $E$ is a subgroup of $\Ab(G')$. We can form the abelian group $H$ by taking the direct product of $V$ and $\Ab(G')$ and then associating their $E$ subgroups. We have
$$|H| = |\Ab(\pi_1(\G'))| \cdot \frac{|V|}{|E|.}$$
Clearly $H$ is an abelian quotient of $G = \pi_1(\G)$. On the other hand, it is not difficult to see that the kernel of the map $G \rightarrow H$ is contained in $[G, G]$. Thus $H = \Ab(G)$. The statement of the lemma now follows by induction.
\end{proof}

\begin{lemma} \label{LEM EQNS}
Let $G$ be a finitely generated, virtually free, and para-$C_p^{*n}$ group. Then $G$ is the fundamental group of a finite tree of groups $\G$. The vertex groups of $\G$ are finite direct products of $C_p$'s and all of the edge groups are proper subgroups of the vertex groups they join. Moreover, if $V_1, V_2, \ldots, V_k$ and $E_1, E_2, \ldots, E_{k-1}$ are any enumerations of the vertex groups and edge groups of $\G$, then
$$\left( \frac{|V_1|}{|E_1|} \right) \left( \frac{|V_2|}{|E_2|} \right) \cdots \left( \frac{|V_{k-1}|}{|E_{k-1}|} \right) |V_k| = p^n$$
and
$$\frac{1}{|E_1|} + \frac{1}{|E_2|} + \cdots + \frac{1}{|E_{k-1}|} - \frac{1}{|V_1|} - \frac{1}{|V_2|} - \cdots - \frac{1}{|V_k|} = n - 1 - \frac{n}{p}$$
\end{lemma}

\begin{proof}
Since $G$ is finitely generated and virtually free, it follows from A.~Karrass, A.~Pietrowski, and D.~Solitar \cite{KPS73} (see also G.~P.~Scott \cite{S74}) that $G$ is the fundamental group of a graph of groups $\G$ in which all vertex groups are finite. $\G$ must be a tree as otherwise $G = \pi_1(\G)$ maps onto $\Z$, contradicting the fact that $G$ is para-$C_p^{*n}$. Since $G$ is para-$C_p^{*n}$ each of the vertex groups of $\G$ must be a finite $p$-group. Since $[G, G]$ is torsion free the vertex groups of $\G$ must be abelian. The abelianization of $G$ is the $n$-fold direct product $C_p^n$ and therefore $g^p \in [G, G]$ for all $g \in G$. It follows that the vertex groups must be finite direct products of $C_p$'s. If any edge group is equal to one of the vertex groups in its endpoints, we remove that edge and associate its two vertex groups to one another. This allows us to assume that $\G$ has the property that each edge group is properly contained in the vertex groups at each of its ends. The commutator subgroup $[G, G]$ is torsion free and virtually free. Thus by Stallings \cite{S68}, it must in fact be free. By Lemma \ref{MainLemma1}, $[G, G]$ and $[C_p^{*n}, C_p^{*n}]$ have the same rank and same index in $G$ and $C_p^{*n}$, respectively. By A.~Karrass, A.~Pietrowski, and D.~Solitar \cite{KPS73}, knowledge of the rank and index of a free subgroup of a fundamental group of a graph of groups leads to restrictions on the orders of the vertex groups and edge groups in the graph of groups. In our case, this restriction manifests as the following equation:
$$\frac{1}{|E_1|} + \frac{1}{|E_2|} + \ldots + \frac{1}{|E_{k-1}|} - \frac{1}{|V_1|} - \frac{1}{|V_2|} - \cdots - \frac{1}{|V_k|} = n - 1 - \frac{n}{p},$$
where $V_1, V_2, \ldots, V_k$ are the vertex groups of $\G$ and $E_1, E_2, \ldots, E_{k-1}$ are the edge groups of $\G$. Finally, since $\Ab(\pi_1(\G)) = \Ab(G) = C_p^n$ we have
$$\left( \frac{|V_1|}{|E_1|} \right) \left( \frac{|V_2|}{|E_2|} \right) \cdots \left( \frac{|V_{k-1}|}{|E_{k-1}|} \right) |V_k| = p^n$$
by Lemma \ref{LEM ABTREE}.
\end{proof}

\begin{corollary} \label{COR KN}
If $G$ and $\G$ are as in the previous lemma and $k \geq n$, then $k = n$ and $G = C_p^{*n}$.
\end{corollary}

\begin{proof}
We first fix an enumeration of the vertices and edges of $\G$ by applying Lemma \ref{LEM INDX}. Since all of the edge groups are properly contained in their corresponding vertex groups, we have $\frac{|V_i|}{|E_i|} \geq p$ for each $1 \leq i < k$. We also have $|V_k| \geq p$. Therefore
$$\left( \frac{|V_1|}{|E_1|} \right) \left( \frac{|V_2|}{|E_2|} \right) \cdots \left( \frac{|V_{k-1}|}{|E_{k-1}|} \right) |V_k| \geq p^k.$$
If $k \geq n$ then by the previous lemma we have $k = n$, $|V_n| = p$, and $\frac{|V_i|}{|E_i|} = p$ for each $1 \leq i < n$. If $|V_i| = p$ for each $1 \leq i < n$ then all of the vertex groups are $C_p$ and all of the edge groups are trivial (since they are proper subgroups). In this case $G = C_p^{*n}$. Towards a contradiction, suppose there is some $i$ with $|V_i| > p$. Let $V_i$ be the vertex closest to $V_n$ with $|V_i| > p$. By Lemma \ref{LEM INDX}, $E_i$ has an endpoint closer to $V_k$. This other endpoint must therefore have cardinality $p$. Since $E_i$ is a proper subgroup of that vertex group, $E_i$ must be trivial and hence $|E_i| = 1$. However, we already pointed out that $\frac{|V_i|}{|E_i|} = p$, a contradiction.
\end{proof}

\begin{theorem}
Let $p$ be an odd prime and let $G$ be a group with the same lower central series quotients as $C_p^{*n}$. If $G$ is residually nilpotent, finitely generated, and virtually free and $n \leq p$, then $G = C_p^{*n}$.
\end{theorem}

\begin{proof}
Let $\G$ be the tree of groups described in Lemma \ref{LEM EQNS}, and let $k$ be the number of vertices of $\G$. By the previous corollary, we may suppose that $k \leq n$. It will suffice to show that $k = n$.

Let $v_1, v_2, \ldots, v_k$ be the orders of the vertex groups of $\G$ and let $e_1, e_2, \ldots, e_{k-1}$ be the orders of the edge groups of $\G$. By Lemma \ref{LEM EQNS} we have
$$\frac{1}{e_1} + \frac{1}{e_2} + \ldots + \frac{1}{e_{k-1}} - \frac{1}{v_1} - \frac{1}{v_2} - \cdots - \frac{1}{v_k} = n - 1 - \frac{n}{p}.$$
We introduce a dummy variable $e_k$ and set $e_k = 1$. The $-1$ on the right hand side can be rewritten $-e_k$ and then be moved to the left hand side to give
$$\sum_{i = 1}^k \left( \frac{1}{e_i} - \frac{1}{v_i} \right) = n - \frac{n}{p}.$$
By Lemma \ref{LEM INDX} we can assume that the vertices and edges are ordered so that $v_i \leq v_k$ for all $1 \leq i \leq k$. Also notice that $e_i < v_i$ for $1 \leq i \leq k$ since all of the edge groups are proper subgroups of the vertex groups they join. Multiplying both sides of the above equation by $v_k$ gives
$$\sum_{i = 1}^k \left( \frac{v_k}{e_i} - \frac{v_k}{v_i} \right) = n (p-1) \left( \frac{v_k}{p} \right).$$

If $v_k = p$, then by maximality of $v_k$ we have $v_i = p$ for all $i$. Since the edge groups are all proper, it follows $e_i = 1$ for all $i$. Then by Lemma \ref{LEM ABTREE} we have that $k = n$ and $G = C_p^{*n}$. Now suppose that $v_k > p$. It follows that the right hand side of the above expression is divisible by $p$. So we have
$$0 \equiv \sum_{i = 1}^k \left( \frac{v_k}{e_i} - \frac{v_k}{v_i} \right) \equiv - \left| \{ i : 1 \leq i \leq k, \ v_i = v_k\} \right| \mod p.$$
Set $S = \{ i : 1 \leq i \leq k, \ v_i = v_k\}$. We clearly have $k \in S$ so $|S| \neq 0$. However, $|S| \leq k \leq n \leq p$. Therefore $|S| = p$ and $k = n$. Then by Corollary \ref{COR KN} we have $G = C_p^{*n}$.
\end{proof}

\section{Existence of para-$\ffp$ groups} \label{ProofSection2}
Our next two results show that in the case $\Gamma = \ffp$, there exists examples of Type II and III (please see Section \ref{PrelimSection}).
That is, weakly para-$(\ffp)$ groups which are not para-$(\ffp)$ groups exist.
And, further, para-$(\ffp)$ groups which are not isomorphic to $\ffp$ exist.
These two results give Theorem \ref{MainTheorem1} from the introduction.

\begin{theorem} \label{weaklyTheoremExample}
Let $G = \left< a, b : (a[a^p,b])^p, b^p \right>.$
Then $G$ is a weakly para-$(\ffp)$ group of Type II.
\end{theorem}

\noindent {\bf Remark:}
While one can prove this theorem directly, we decided to take a more explorative approach in this proof. For instance, Claim \ref{linearAlgLemma} in the proof of Theorem \ref{weaklyTheoremExample} distinguishes a certain family of groups from $C_p$.
Compare this claim with Theorem 1 in \cite{MS} by Miller and Schupp. As a consequence of Theorem 1 in \cite{MS}, the groups
$$\left< b, a : a[a^n, b], b = w \right>$$
where $n > 1$ and $w$ has exponent sum $0$ on $b$, are trivial.
The major difficulty in the proof of Claim \ref{linearAlgLemma} is obtaining a result that holds for \emph{all} odd primes $p$.

\paragraph{Proof of Theorem \ref{weaklyTheoremExample}.} Let $G = \left< a, b : (a[a^p, b])^p, b^p \right>$. The following technical claim is used to show that $G$ is not residually nilpotent.

\begin{claim} \label{linearAlgLemma}
The element $a$ is nontrivial in $H=\left<a, b : (a[a^p, b]), b^p\right>$.
\end{claim}

\emph{Proof of Claim.}
Let $\phi: H \to C_p$ be the map with $a \mapsto 0$ and $b \mapsto 1$.
Using the Reidemeister-Schreier Method with Schreier basis

$$
\{ 1, b, b^2, \ldots, b^{p-1} \},
$$

\noindent
it is straightforward to show that we get the following presentation for $K:= \ker\phi$:
$$
K = \left< x_1, \ldots, x_p : x_{i_n}^{\epsilon_n p} x_{j_n}^{-\epsilon_n(p-1)} , n=1,\ldots, p\right>,
$$
where $i_n \neq j_n$ and $\epsilon_n \in \{ \pm 1\}$.  Further, the method gives that the generators $x_i$ are all conjugates of $a$ in $H$. If we can show that $K$ is nontrivial, then it will follow that $a \neq 1$ in $H$.

 The abelianization of $K$ is the group 
$\Z^p/N,$
where $N$ is the subgroup generated by $v_1,\ldots, v_p$ vectors in $\Z^p$, each $v_n$ coming from a distinct relation  $x_{i_n}^{\epsilon_n p} x_{j_n}^{-\epsilon_n(p-1)} $.
Form the matrix $A$ with the vectors $v_1,\ldots, v_p$ as rows.
With careful inspection of the presentation $K$ obtained from the Reidemeister-Schreier Method, we see that for $p > 3$ the matrix $A$ is a row-permutation of
$$
\begin{pmatrix}
-x & y & & & & & & &\\
-y & 0& x & & & & & &\\
 & -x &0 & y & & & & &\\
 & & -y &0 & x & & & &&\\
 & & & \ddots &\ddots &\ddots & & &\\
  & & & & -y & 0& x & &\\
 & & & & & -x& 0& y&0 \\
 & & & & & & 0& -x&y \\
 & & & & & & -y&0 &x 
\end{pmatrix},
$$
where $x = p$ and $y = p-1$ and all the blank spaces are assumed to be filled with zeros. From this form we compute
$$
|\det(A)| = p^p - (p-1)^p,
$$
for $p > 3$. The case $p = 3$ needs to be computed separately (this equality is also true for $p =2 $).
In the case $p = 3$, the matrix $A$ is a row-permutation of
$$
\begin{pmatrix} -3& 2 & 0 \\
0 & -3 & 2 \\
-2 & 0 & 3 \end{pmatrix},
$$
so $|\det(A)| = 19 = 3^3 - 2^3$.
Hence $|\det(A)| \neq 1$ so $A^{-1}$ cannot be integral. So $\left<v_1, \ldots, v_p\right> \neq \Z^p$.
It follows that $K$ cannot be the trivial group, as it has nontrivial abelianization, and the proof of the claim is complete.

\begin{claim}
The group $G$ is not residually nilpotent.
\end{claim}

\emph{Proof of Claim.}
The map $\Phi: G \to \ffp = \left< \alpha, \beta : \alpha^p , \beta^p \right>$ given by $a \mapsto \alpha$ and $b \mapsto \beta$ is onto. We will show that the map $\Phi$ is not one-to-one, and hence as $G$ is weakly para-$(\ffp)$ it follows that $G$ cannot be residually nilpotent.

Indeed, we will show that $a^p \neq 1$ in $G$, and hence $\ker \Phi \neq 1$.
Let $F = F(a,b)$ be the free group of rank $2$.
If $a^p = 1$ in $G$, then $a^p$ is in the normal group generated by $(a[a^p,b])^p$ and $b^p$ in $F$.
Thus, $a^p$ is contained in the normal group generated by $(a[a^p, b])$ and $b^p$.
It follows that $a = 1$ in the group 
$$
H=\left< a, b : (a[a^p, b]), b^p \right>,
$$
contradicting Claim \ref{linearAlgLemma}. 
This finishes the proof of the claim.

\begin{claim} \label{claim1}
The group $G$ is weakly para-$(\ffp)$.
\end{claim}

\emph{Proof of Claim.} Let $F = F(a,b)$ be a free group of rank two.
Let $\Phi : F \to F$ be the map given by $a \mapsto a[a^p, b]$ and $b \mapsto b$.
This gives a well defined map $\Phi': F_i \to F_i$ where $F_i := F/\gamma_i(F)$.
This is an epimorphism by Lemma \ref{generatorlemma}.
Since finitely generated nilpotent groups are Hopfian, $\Phi'$ must be an isomorphism.
Therefore the induced map on $\ffp/\gamma_i(\ffp) \to G/\gamma_i(G)$ is an isomorphism, as claimed.

As $G$ is weakly para-$(\ffp)$ and is not residually nilpotent, the proof of Theorem \ref{weaklyTheoremExample} is complete. \qed

\begin{theorem} \label{stronglyTheoremExample}
Let $G = \left< a, b :  (a [b, a])^p, b^p \right>.$
Then $G/ \cap_{k=1}^\infty \gamma_k(G)$ is a para-$(\ffp)$ group of Type III.
\end{theorem}

\paragraph{Proof of Theorem \ref{stronglyTheoremExample}.}
Let $H = G/\cap_{k=1}^\infty \gamma_k(G)$.
Our first claim gives a technical result that will aid us in distinguishing $H$ from $\ffp$.

\begin{claim} \label{infOrderLemma}
Let $\gamma \in \ffp = \left< a, b : a^p, b^p\right> $ be a nontrivial element not equal to a power of $b$.
Then the element $\gamma [b,\gamma]$ has infinite order.
\end{claim}

\emph{Proof of Claim.}
By Lemma \ref{NormalFfpForm}, we write $\gamma$ as in Equation \ref{normalform}.
For some $n\neq0$, one of the following happens:
\begin{enumerate}
\item[Case 1.] $g_1 = a^m$ and $g_k = a^n$ for some $m,n\neq0$:
The element $\gamma[b, \gamma]$ is conjugate to
\[
b^{-1} \gamma^{-1} b \gamma^2.
\]
The element $\gamma^2$ either has normal form
$$
g_1 \cdots g_{k-m} \alpha g_{m} \cdots g_k,
$$
where $\alpha \in \{ a, a^{2}, \ldots, a^{p-1}, b, b^{2}, \ldots, b^{p-1} \}$ or $\gamma^2$ has normal form $\alpha \in \{ 1, a, a^{2}, \ldots, a^{p-1} \}$. The equality $\gamma^2 = 1$ is impossible in $\ffp$ as $p \neq 2$. 
It follows that $
b^{-1} \gamma^{-1} b \gamma^2$ has normal form starting with $b^{-1}$ and ending with a power of $a$, thus $\gamma[b,\gamma]$ must have infinite order.

\item[Case 2.] $g_1 = b^m$ and $g_k = b^n$: In this case, because $\gamma$ is not a power of $b$, we have $k > 2$ and
\begin{eqnarray*}
\gamma [b, \gamma] &=& \gamma  b^{-1} \gamma^{-1} b \gamma \\
&=& b^m g_2 \cdots b^n b^{-1} b^{-n} \cdots b^{-m} b b^m \cdots g_{k-1} b^n.
\end{eqnarray*}
Hence, $\gamma [b, \gamma]$ is conjugate to
$$
 b^{m+n} g_{2} \cdots g_{k-1} b^{-1} g_{k-1}^{-1} \cdots g_{2}^{-1} b g_2 \cdots g_{k-1}.
$$
If $m+n \neq 0$, then this normal form begins with a power of $b$ and ends with a power of $a$. So $(\gamma[b,\gamma])^n$ is never trivial.
Otherwise, if $m+n = 0$, we appeal to Case 1, replacing $\gamma$ with $g_2 \cdots g_{k-1}$.

\item[Case 3.] $g_1 = b^m$ and $g_k = a^n$: In this case we have
\begin{eqnarray*}
\gamma [b, \gamma] &=& \gamma  b^{-1} \gamma^{-1} b \gamma \\
&=& b^m \cdots a^n b^{-1} a^{-n} \cdots b^{-m} b b^m \cdots a^n \\
&=&b^m \cdots a^n b^{-1} a^{-n} \cdots g_{2}^{-1} b g_2 \cdots a^n,
\end{eqnarray*}
which is in normal form. Since this form begins with a power of $b$ and ends with a power of $a$, we see that $\gamma[b,\gamma]$ has infinite order.

\item[Case 4.] $g_1 = a^m$ and $g_k = b^n$: In this final case we have
\begin{eqnarray*}
\gamma [b, \gamma] &=& \gamma  b^{-1} \gamma^{-1} b \gamma \\
&=& a^m \cdots b^n b^{-1} b^{-n} \cdots a^{-m} b a^m \cdots b^n \\
&=& a^m \cdots g_{k-1} b^{-1} g_{k-1}^{-1} \cdots a^{-m} b a^m \cdots b^n.
\end{eqnarray*}
Just as in Case 3, an element with this normal form must have infinite-order.
\end{enumerate}

\begin{claim} \label{paraffpclaim}
The group $H$ is para-$(\ffp)$.
\end{claim}

\emph{Proof of Claim.}
Let $\Phi : F \to F$ be the map given by $a \mapsto a[b, a]$ and $b \mapsto b$, then follow the proof of Claim \ref{claim1}, to conclude that $G$ is weakly para-$(\ffp)$.
Since $H/\gamma_k(H) = G/\gamma_k(G)$ for all $k$ and $H$ is residually nilpotent, the claim is shown.
\qed

\begin{claim}
The group $H$ is not isomorphic to $\ffp$.
\end{claim}

\emph{Proof of Claim.}
Suppose, for the sake of a contradiction, that $H= \left< \alpha, \beta : \alpha^p, \beta^p \right>$.
By abuse of notation, let $a$ and $b$ be the images of $a$ and $b$ under the projection $G \to H$.
Then by switching $\alpha$ and $\beta$ if necessary and by applying Lemma \ref{NormalFfpForm}, we conclude that $\beta$ is conjugate to $b$.
Hence by suitably relabeling variables we may assume, without loss of generality, that $H = \left< \alpha, b : \alpha^p, b^p \right>$. But we also have that $H$ is a quotient of $G = \left< a, b : (a[b,a])^p, b^p \right>$, so the element $a[b,a]$ has order $p$ in $H$. Further, $a$ is nontrivial and not a power of $b$. But then $a[b,a]$ must have infinite order by Claim \ref{infOrderLemma}. \qed

\subsection{Generalizing to $C_{p^l}*C_{p^k}$}
\label{generalizationSection}
We outline here a proof of the generalization of Theorem \ref{MainTheorem1} to $C_{p^l}*C_{p^k}$.

\begin{theorem} \label{MainTheorem2}
Let $p$ be an odd prime and $l,k$ be natural numbers.
There exist rank two groups $G_1$ and $G_2$, both not isomorphic to $C_{p^l}*C_{p^k}$, such that
 $$G_1/\gamma_i(G_1) \cong G_2/\gamma_i(G_2) \cong C_{p^l}*C_{p^k} / \gamma_i(C_{p^l}*C_{p^k})$$
 for all $i$. Further, $G_1$ is residually nilpotent and $G_2$ is not.
\end{theorem}

\begin{proof}
For $G_1$, set $G = \left< a, b : (a[b,a])^{p^l}, b^{p^k} \right>$.
The proof of Theorem \ref{stronglyTheoremExample} can be followed almost verbatim to conclude that $G_1 =  G/\cap_{i=1}^\infty \gamma_i(G)$ has the desired properties.

For $G_2$, suppose $l < k$.
Set $G_2 = \left< a, b : (a[a^{p^l}, b])^{p^l}, b^{p^k} \right>$ and $H =\left< a, b : a[a^{p^l}, b], b^{p^l} \right>$.
Now follow the proof of Claim \ref{linearAlgLemma} to conclude that $a \neq 1$ in $H$.
Thus, $a \neq 1$ in $H' = \left< a, b : a[a^{p^l}, b], b^{p^k} \right>$.
From this, we conclude that $a^{p^l} \neq 1$ in $G_2$.
Following the rest of the proof of Theorem \ref{weaklyTheoremExample} shows that $G_2$ has the desired properties.
\end{proof}

\section{Final remarks} \label{finalremarks}

This note is partially motivated by the desire to understand topological applications to Cochran and Harvey's paper \cite{cochran}.
Recall that the \emph{$p$-lower central series} $\{G_{p,n} \}$ is the fastest descending central series with successive quotients that are $\Z_p$-vector spaces.
In 1965, J. Stallings gave homological conditions that ensure that a group homomorphism induces an isomorphism modulo any term of the $p$-lower central series (see \cite{S65}).
The result of the previous section addresses the existence of a converse to Stalling's theorem, as we have constructed groups $G$ that are isomorphic to $\ffp$ modulo any term of the $p$-lower central series but are not $\ffp$.
However, we have been unable to show that our residually nilpotent example is finitely presented.
In light of this, we ask the following question.

\begin{question}
Does there exist a finitely presented para-$(\ffp)$ group which is not $\ffp$?
\end{question}


\begin{thebibliography}{9}

\bibitem{musings} G. Baumslag, \emph{Musings on Magnus. The mathematical legacy of Wilhelm Magnus: groups, geometry and special functions} (Brooklyn, NY, 1992), 99--106,
Contemp. Math., 169, Amer. Math. Soc., Providence, RI, 1994.

\bibitem{baumslag-2005}
G. Baumslag, \emph{Parafree groups. Infinite groups: geometric, combinatorial and dynamical aspects}, 1--14, Progr. Math., 248, BirkhŠuser, Basel, 2005.

\bibitem{baumslag-1967} G. Baumslag, 
\emph{Groups with the same lower central sequence as a relatively free group I, the groups.}, Trans. Amer. Math. Soc. {\bf 129} (1967), 308-321.

\bibitem{baumslag-1968} G. Baumslag,
\emph{More groups that are just about free.} Bull. Amer. Math. Soc. {\bf 74} (1968), 752--754.

\bibitem{baumslag-1969} G. Baumslag, 
\emph{Groups with the same lower central sequence as a relatively free group II, properties.}, Trans. Amer. Math. Soc. {\bf 129} (1969), 507-538.

\bibitem{bou} K. Bou-Rabee, \emph{Parasurface groups}. To appear in Pac. J. Math.

\bibitem{bridson-grunewald} M. Bridson and F. Grunewald, \emph{Grothendieck's problems concerning profinite completions and representations of groups.} (English summary) Ann. of Math. (2) {\bf 160} (2004), no. 1, 359--373.

\bibitem{cochran} T.~Cochran, S.~Harvey, \emph{Homology and derived $p$-series of groups.} J. Lond. Math. Soc. (2) 78 (2008), no. 3, 677--692.

\bibitem{GAP} The GAP Group, \emph{GAP -- Groups, Algorithms, and Programming}, Version 4.4.12; 2008. (http://www.gap-system.org)

\bibitem{harpe-2000}  
P. de La Harpe,  \emph{Topics in Geometric Group Theory}, Chicago Lectures in Mathematics, Chicago 2000.

\bibitem{KPS73}
A.~Karrass, A.~Pietrowski, and D.~Solitar,
\emph{Finite and infinite cyclic extensions of free groups}, J. Austral. Math. Soc. 16 (1973) pp. 458-466.

\bibitem{L1} S. Liriano, \emph{Algebraic geometric invariants of parafree groups.} Internat. J. Algebra Comput. 17 (2007), no. {\bf 1}, 155--169.

\bibitem{M1} W. Magnus, \emph{\"Uber freie Faktorgruppen und freie Untergruppen gegebener Gruppen}, Monatsh. Math. {\bf 47} (1939), 307Ð313. 

\bibitem{MKS} W. Magnus, A. Karrass, and D. Solitar, \emph{Combinatorial group theory. Presentations of groups in terms of generators and relations.} Reprint of the 1976 second edition. Dover Publications, Inc., Mineola, NY, 2004.

\bibitem{MS} C.F. Miller and P. Schupp, \emph{Some presentations of the trivial group. In Groups, Languages and Automata} (R. Gilman, Ed.), Contemporary Mathematics, Vol. 250, American Mathematical Society (1999) 113--115.

\bibitem{S74} G.~P.~Scott, \emph{An embedding theorem for groups with a free subgroup of finite index}, Bull. London Math. Soc., 6 (1974), 304-306.

\bibitem{S65} J. Stallings, \emph{Homology and central series of groups.} Journal Algebra, 2:170--181, 1965.

\bibitem{S68} J. Stallings, \emph{On torsion-free groups with infinitely many ends.} Annals of Mathematics, 88 (1968), no. {\bf 2}, 312--334.

\end{thebibliography}
\end{document}